\documentclass[12pt]{amsart}
\usepackage{graphicx}
\usepackage[headings]{fullpage}
\usepackage{amssymb,epic,eepic,epsfig,amsbsy,amsmath,amscd}
\numberwithin{equation}{section}
                        \theoremstyle{plain}
\usepackage{mathrsfs}

\newcommand\no[1]{}

\newtheorem{theorem}{Theorem}[section]

\newtheorem{lemma}[theorem]{Lemma}

\newtheorem{proposition}[theorem]{Proposition}

\theoremstyle{definition}
\newtheorem{remark}[theorem]{Remark}

\newcommand{\p}{\partial}

\def\BC{\mathbb C}

\def\BZ{\mathbb Z}
\def\BR{\mathbb R}

\def\BH{\mathbb H}

\def\CP{\mathcal P}

\def\CT{\mathcal T}

\def\fb{\mathfrak b}

\def\la{\langle}
\def\ra{\rangle}

\DeclareMathOperator{\tr}{\mathrm tr}

\def\ve{\varepsilon}
\def\be { \begin{equation} }
\def\ee { \end{equation} }

\begin{document}

\title[Twisted Alexander polynomials of $2$-bridge knots]
{Twisted Alexander polynomials of $2$-bridge knots \\
for parabolic representations}

\author[Takayuki Morifuji]{Takayuki Morifuji}
\address{Department of Mathematics, Hiyoshi Campus, Keio University, Yokohama 223-8521, Japan}
\email{morifuji@z8.keio.jp}

\author[Anh T. Tran]{Anh T. Tran}
\address{Department of Mathematics, The Ohio State University, Columbus, OH 43210, USA}
\email{tran.350@osu.edu}


\begin{abstract}
In this paper we show that the twisted Alexander polynomial associated to 
a parabolic representation determines fiberedness and genus of a wide class 
of $2$-bridge knots. As a corollary we give an affirmative answer to 
a conjecture of Dunfield, Friedl and Jackson for infinitely many hyperbolic knots. 
\end{abstract}

\maketitle

\section{Introduction}

The twisted Alexander polynomial was introduced by Lin \cite{Lin01-1} 
for knots in the $3$-sphere and by Wada \cite{Wada94-1} for finitely presentable groups. 
It is a generalization of the classical Alexander polynomial and gives a powerful 
tool in low dimensional topology. A theory of twisted Alexander polynomials was 
rapidly developed during the past ten years. Among them, one of the most important progress 
is the determination of fiberedness \cite{FV11-1} and genus (the Thurston norm) \cite{FV12-1} 
of knots by the collection of the twisted Alexander polynomials corresponding to all 
finite-dimensional representations. 
For literature on other applications and related topics, we refer to the survey paper by
Friedl and Vidussi \cite{FV10-1}.

Let $K$ be a knot in $S^3$ and $G_K$ its knot group. Namely it is 
the fundamental group of the complement of $K$ in $S^3$, i.e. $G_K=\pi_1(S^3\backslash K)$. 
In this paper, we consider the twisted Alexander polynomial 
$\Delta_{K,\rho}(t)\in \BC[t^{\pm1}]$ associated to a parabolic representation 
$\rho:G_K\to SL_2(\BC)$. 
A typical example is the holonomy representation 
$\rho_0:G_K\to SL_2(\BC)$ of a hyperbolic knot $K$, 
which is a lift of a discrete faithful representation 
$\bar{\rho}_0:G_K\to PSL_2(\BC)\cong\text{Isom}^+(\BH^3)$ such that 
$\BH^3/\bar{\rho}_0(G_K)\cong S^3\backslash K$ 
where $\BH^3$ denotes the upper half space model of the hyperbolic $3$-space (see \cite{Th}). 
In \cite{DFJ}, Dunfield, Friedl and Jackson numerically computed the twisted Alexander 
polynomial $\CT_K(t)=\Delta_{K,\rho_0}(t)$, which is called the hyperbolic torsion polynomial, 
for all hyperbolic knots of 15 or fewer crossings. 
Based on these huge computations, they conjectured that the hyperbolic torsion 
polynomial determines the knot genus and moreover, the knot is fibered 
if and only if $\CT_K(t)$ is a monic polynomial. 
This conjecture seems to be nice because we would say fiberedness and genus of a given 
knot by the twisted Alexander polynomial associated to a single representation. 
However it is widely open except for the hyperbolic twist knots \cite{Mo}. 

The purpose of this paper is to show that the above conjecture is true for a wide 
class of $2$-bridge knots. Since $2$-bridge knots are alternating, their 
fibering and genus can be determined by the Alexander polynomial \cite{Cr, Mu}. 
However there seems to be no a priori reason that the same must 
be true for the hyperbolic torsion polynomial. 

For a prime $p$ and an integer $a$ between $1$ and $p-1$, we say that $a$ is a primitive root modulo $p$ if it is a generator of the cyclic group $(\BZ/p\BZ)^*$. Let $\CP_2$ be the set of all odd primes $p$ such that $2$ is a primitive root modulo $p$. 
Note that all primes $p=2q+1$ such that $q$ is a prime $\equiv 1 \pmod{4}$ are contained in $\CP_2$, see e.g. \cite[Theorem 5.6]{LeV}. 
Then we have the following: 

\begin{theorem}\label{thm:DFJ-conjecture}
Let $K$ be the knot $J(k,2n)$ as in Figure 1, where $k>0$ and $n\in \BZ$. 
For all hyperbolic knots $K$, the hyperbolic torsion polynomial 
$\CT_K(t)$ determines the genus of $K$. Moreover for
$k=2m+1$, or $k=2$ (twist knot), or $k=2m$ and $|4mn-1| \in \CP_2$, 
the knot $J(k,2n)$ is fibered if and only if $\CT_K(t)$ is monic.
\end{theorem}


As mentioned above, the holonomy representation $\rho_0$ is parabolic, 
so that Theorem~\ref{thm:DFJ-conjecture} 
is an immediate corollary of the following theorem. 

\begin{theorem} \label{thm:main-theorem}

Suppose $\rho:G_K\to SL_2(\BC)$ is a parabolic representation of $K=J(k,2n)$. 
Then we have 
\begin{enumerate}
\item
$\Delta_{K,\rho}(t)$ determines the genus of $J(k,2n)$, and

\item
$\Delta_{K,\rho}(t)$ determines the fiberedness of $J(k,2n)$ if $k=2m+1$, or $k=2$ (twist knot), or $k=2m$ and $|4mn-1| \in \CP_2$.
\end{enumerate}
\end{theorem}

\begin{remark}
(1) Suppose $k=2m$ and $n>0$. Then $4mn-1 \in \CP_2$ if $4mn-1$ is a prime and $2mn-1$ is a prime $\equiv 1 \pmod{4}$.

(2) It is known that the conjugacy classes of parabolic representations into $SL_2(\BC)$ of the knot $J(2m,2n)$ can be described as the zero locus of an integral polynomial in one variable. The condition $|4mn-1| \in \CP_2$ in Theorem~\ref{thm:main-theorem} (hence Theorem~\ref{thm:DFJ-conjecture}) assures the irreducibility over $\BZ$ of this polynomial, see Section \ref{condition}. We do not know whether Theorem~\ref{thm:main-theorem} (2) holds true for every integers $m$ and $n$.    
\label{remark}
\end{remark}

This paper is organized as follows. 
In Section 2, 
we study non-abelian representations of the knot $J(k,2n)$ 
and give an explicit formula of the defining equation of the representation space. 
In Section 3, we investigate parabolic representations of $J(k,2n)$. 
In Section 4, we quickly review the definition of the twisted Alexander polynomial 
and some related work on fibering and genus of knots. 
In particular, we calculate the coefficients 
of the highest and lowest degree terms of the twisted Alexander polynomial associated to a non-abelian representation of $J(k,2n)$ 
 and give the proof 
of Theorem~\ref{thm:main-theorem} (1). 
In Section 5, we discuss the fibering problem and prove Theorem~\ref{thm:main-theorem} (2). 

\section{Non-abelian representations}

\label{nonab}

Let $K=J(k,l)$ be the knot as in Figure 1. 
Note that $J(k,l)$ is a knot if and only if $kl$ is even, and is the trivial knot if $kl=0$. 
Furthermore, $J(k,l)\cong J(l,k)$ and 
$J(-k,-l)$ is the mirror image of $J(k,l)$. Hence, in the following, we consider $K=J(k,2n)$ for $k>0$ and $|n|>0$. 
When $k=2$, $J(2,2n)$ presents the twist knot. 

In this section we explicitly calculate the defining equation of the non-abelian representation space of $J(k,2n)$.

By \cite{HS} the knot 
group of $K=J(k,2n)$ is presented by $G_K = \la a,b~|~w^na=bw^n \ra$, 
where
$$w = 
\begin{cases} 
(ba^{-1})^m(b^{-1}a)^m, & k=2m,\\
(ba^{-1})^mba(b^{-1}a)^m, & k=2m+1.
\end{cases}$$

\begin{figure}
\setlength{\unitlength}{0.09mm}
\thicklines{
\begin{picture}(300,460)(80,-20)
\put(0,0){\line(0,1){440}}
\put(0,0){\line(1,0){125}}
\put(100,100){\line(1,0){25}}
\put(100,100){\line(0,1){125}}
\put(125,-25){\line(1,0){150}}
\put(125,-25){\line(0,1){150}}
\put(125,125){\line(1,0){150}}
\put(275,-25){\line(0,1){150}}
\put(100,225){\line(1,0){50}}
\put(150,225){\line(0,1){25}}
\put(125,250){\line(1,0){150}}
\put(125,250){\line(0,1){150}}
\put(125,400){\line(1,0){150}}
\put(275,400){\line(0,-1){150}}
\put(0,440){\line(1,0){150}}
\put(150,440){\line(0,-1){40}}
\put(190,30){{\large$l$}}
\put(185,310){\large{$k$}}
\put(250,400){\line(0,1){40}}
\put(250,440){\line(1,0){150}}
\put(400,440){\line(0,-1){440}}
\put(275,0){\line(1,0){125}}
\put(250,225){\line(1,0){50}}
\put(250,225){\line(0,1){25}}
\put(300,100){\line(0,1){125}}
\put(300,100){\line(-1,0){25}}
\end{picture}
}
\caption{The knot $K=J(k,l)$. Here $k>0$ and $l=2n~(n \in\BZ)$ denote 
the numbers of half twists in each box. Positive numbers correspond 
to right-handed twists and negative numbers correspond to left-handed 
twists respectively. }
\end{figure}
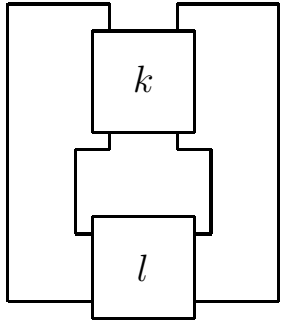

Let $\{S_i(z)\}_i$ be the sequence of Chebyshev polynomials defined by $S_0(z)=1,\,S_1(z)=z$, and $S_{i+1}(z)=zS_i(z)-S_{i-1}(z)$ for all integers $i$.

The following lemmas are standard, see e.g \cite[Lemma 2.3]{Tr} and \cite[Lemma 2.2]{Tr2}.

\begin{lemma}
One has $S^2_i(z)-zS_i(z)S_{i-1}(z)+S^2_{i-1}(z)=1$.
\label{S}
\end{lemma}

\begin{lemma}
Suppose the sequence $\{M_i\}_i$ of $2 \times 2$ matrices satisfies 
the recurrence relation $M_{i+1}=zM_i-M_{i-1}$ for all integers $i$. Then 
\begin{eqnarray}
M_i &=& S_{i-1}(z)M_1-S_{i-2}(z)M_0 \label{1}\\
&=&S_i(z)M_0-S_{i-1}(z)M_{-1} \label{-1}.
\end{eqnarray}
\end{lemma}

A representation $\rho:G_K\to SL_2(\BC)$ is called non-abelian if 
$\rho(G_K)$ is a non-abelian subgroup of $SL_2(\BC)$. 
Taking conjugation if necessary, we can assume that $\rho$ 
has the form
\begin{equation}
\rho(a)=A=\left[ \begin{array}{cc}
s & 1\\
0 & s^{-1} \end{array} \right] \quad \text{and} \quad \rho(b)=B=\left[ \begin{array}{cc}
s & 0\\
2-y & s^{-1} \end{array} \right]
\label{nonabelian}
\end{equation}
where $(s,y) \in \BC^* \times \BC$ satisfies the matrix equation $W^nA-BW^n=0$. 
Here $W=\rho(w)$. It can be easily checked that $y=\tr AB^{-1}$ holds.
Let $x=\tr A=\tr B=s+s^{-1}$.


\begin{lemma}
One has $$WA-BW=\left[ \begin{array}{cc}
0 & \alpha_{k}(x,y)\\
(y-2)\alpha_{k}(x,y) & 0 \end{array} \right]$$ where
$$\alpha_{k}(x,y)=
\begin{cases}
1-(y+2-x^2)S_{m-1}(y) \left( S_{m-1}(y) - S_{m-2}(y) \right), & k=2m,\\
1+(y+2-x^2)S_{m-1}(y) \left( S_m(y) - S_{m-1}(y) \right), & k=2m+1.
\end{cases}
$$
\label{wabw}
\end{lemma}

\begin{proof} We first note that, by the Cayley-Hamilton theorem, $M^{i+1}=(\tr M) M^i - M^{i-1}$ for all matrices $M \in SL_2(\BC)$ and all integers $i$.

If  $k=2m$ then by applying \eqref{1} twice, we have
\begin{eqnarray*}
WA&=& (BA^{-1})^m(B^{-1}A)^mA \\
&=& S^2_{m-1}(y) BA^{-1}B^{-1}AA -S_{m-1}(y)S_{m-2}(y)(BA^{-1}A+B^{-1}AA)+S^2_{m-2}(y)A.
\end{eqnarray*}
Similarly,
\begin{eqnarray*}
BW &=& B(BA^{-1})^m(B^{-1}A)^m \\
&=& S^2_{m-1}(y) BBA^{-1}B^{-1}A  -S_{m-1}(y)S_{m-2}(y)(BBA^{-1}+BB^{-1}A) + S^2_{m-2}(y)B.
\end{eqnarray*}
Hence, by direct calculations using \eqref{nonabelian}, we obtain
\begin{eqnarray*}
WA-BW &=& S^2_{m-1}(y) (BA^{-1}B^{-1}AA-BBA^{-1}B^{-1}A)+ S^2_{m-2}(y)(A-B) \\
&& - \, S_{m-1}(y)S_{m-2}(y)(BA^{-1}A-BBA^{-1}+B^{-1}AA-BB^{-1}A) \\
&=& \left[ \begin{array}{cc}
0 & \alpha_{k}(x,y)\\
(y-2)\alpha_{k}(x,y) & 0 \end{array} \right] 
\end{eqnarray*}
where  
$$
\alpha_{k}(x,y)= (s^{-2}+1+s^2-y)S^2_{m-1}(y)-(s^{-2}+s^2)S_{m-1}(y)S_{m-2}(y)+S^2_{m-2}(y).
$$
Since $S^2_{m-1}(y)-yS_{m-1}(y)S_{m-2}(y)+S^2_{m-2}(y)=1$ (by Lemma \ref{S}) and $x=s+s^{-1}$, 
$$\alpha_k(x,y)=1-(y+2-x^2)S_{m-1}(y) \left( S_{m-1}(y) - S_{m-2}(y) \right).$$

If $k=2m+1$ then by applying \eqref{-1} twice, we have 
\begin{eqnarray*}
WA&=&  (BA^{-1})^mBA(B^{-1}A)^mA \\
&=& S^2_{m}(y)BAA-S_{m}(y)S_{m-1}(y) ( (BA^{-1})^{-1}BAA+BA(B^{-1}A)^{-1}A) \\
&&+ \, S^2_{m-1}(y) (BA^{-1})^{-1}BA(B^{-1}A)^{-1}A\\
&=& S^2_{m}(y) BAA - S_{m}(y)S_{m-1}(y) (A^3+B^2A) +S^2_{m-1}(y) ABA.
\end{eqnarray*}
Similarly,
\begin{eqnarray*}
BW &=& B(BA^{-1})^mBA(B^{-1}A)^m  \\
&=& S^2_{m}(y)BBA-S_{m}(y)S_{m-1}(y) (B(BA^{-1})^{-1}BA+BBA(B^{-1}A)^{-1}) \\
&&+ \, S^2_{m-1}(y) B(BA^{-1})^{-1}BA(B^{-1}A)^{-1}\\
&=& S^2_{m}(y) BBA - S_{m}(y)S_{m-1}(y) (BA^2+B^3) +S^2_{m-1}(y) BAB.
\end{eqnarray*}
Hence, by direct calculations using \eqref{nonabelian}, we obtain
\begin{eqnarray*}
WA-BW &=& S^2_{m}(y) (BAA-BBA)+ S^2_{m-1}(y)(ABA-BAB) \\
&& - \, S_{m}(y)S_{m-1}(y)(A^3-BA^2+B^2A-B^2) \\
&=& \left[ \begin{array}{cc}
0 & \alpha_{k}(x,y)\\
(y-2)\alpha_{k}(x,y) & 0 \end{array} \right] 
\end{eqnarray*}
where 
\begin{eqnarray*}
\alpha_{k}(x,y) &=& S^2_{m}(y)-(s^{-2}+s^2)S_{m}(y)S_{m-1}(y)+(s^{-2}+1+s^{2}-y)S^2_{m-1}(y)\\
&=& 1+(y+2-x^2)S_{m-1}(y) \left( S_m(y) - S_{m-1}(y) \right).
\end{eqnarray*}
This completes the proof of Lemma \ref{wabw}.
\end{proof}

The proof of the following lemma is similar to that of Lemma \ref{wabw}.

\begin{lemma} One has
\begin{eqnarray*}
\tr W =
\begin{cases}
2+(y-2)(y+2-x^2)S^2_{m-1}(y), & k=2m,\\
x^2-y-(y-2)(y+2-x^2)S_m(y)S_{m-1}(y), & k=2m+1.
\end{cases}
\end{eqnarray*}
\label{trace}
\end{lemma}

We are now ready to calculate the expression $W^nA-BW^n$ as follows. Let $\lambda=\tr W$.

\begin{proposition}
One has
$$
W^nA-BW^n = \left[ \begin{array}{cc}
0 & S_{n-1}(\lambda)\alpha_{k}(x,y)-S_{n-2}(\lambda)\\
(y-2) \left( S_{n-1}(\lambda)\alpha_{k}(x,y)-S_{n-2}(\lambda) \right) & 0 \end{array} \right].
$$
\label{Riley}
\end{proposition}

\begin{proof}
By applying \eqref{1} and Lemma \ref{wabw}, we have
\begin{eqnarray*}
W^nA-BW^n &=& S_{n-1}(\lambda)(WA-BW)-S_{n-2}(\lambda)(A-B)\\
&=&  S_{n-1}(\lambda) \left[ \begin{array}{cc}
0 & \alpha_{k}(x,y)\\
(y-2)\alpha_{k}(x,y) & 0 \end{array} \right] - S_{n-2}(\lambda) \left[ \begin{array}{cc}
0 & 1\\
y-2 & 0 \end{array} \right] .
\end{eqnarray*}
The proposition follows.
\end{proof}

Proposition \ref{Riley} implies that the assignment \eqref{nonabelian} gives a non-abelian representation 
$\rho: G_K \to SL_2(\BC)$ 
if and only if $(s,y) \in \BC^* \times \BC$ satisfies the equation 
$$\phi_{k,2n}(x,y):= S_{n-1}(\lambda)\alpha_{k}(x,y)-S_{n-2}(\lambda)=0$$
where $\alpha_k(x,y)$ and $\lambda=\tr W$ are given by the formulas in Lemmas \ref{wabw} and \ref{trace} respectively.

The polynomial $\phi_{k,2n}(x,y)$ is also known as the Riley polynomial \cite{Ri, Le93} of $J(k,2n)$.

\section{Parabolic representations}

A representation $\rho:G_K\to SL_2(\BC)$ is called parabolic if the meridian $\mu$ 
of $K$ is sent to a parabolic element (i.e. $\text{tr}\,\rho(\mu)=2$) of $SL_2(\BC)$ and 
$\rho(G_K)$ is non-abelian. 

Let $K=J(k,2n)$. In this section we will show that if $\rho: G_K \to SL_2(\BC)$ is a parabolic representation of the form $$
\rho(a)=A=\left[ \begin{array}{cc}
1 & 1\\
0 & 1 \end{array} \right] \quad \text{and} \quad \rho(b)=B=\left[ \begin{array}{cc}
1 & 0\\
2-y & 1 \end{array} \right]
$$ where $y$ is a real number satisfying the equation $\phi_{k,2n}(2,y)=0$, then $y>2$.

\begin{lemma}
Suppose $x=2$. Then 
$$\alpha^2_{k}(x,y)-\alpha_{k}(x,y)\lambda+1=
\begin{cases}
(y-2)^3 S^4_{m-1}(y), & k=2m,\\
(y-2) \left( (y-2)S_m(y)S_{m-1}(y)+1 \right)^2, & k=2m+1.
\end{cases}
$$
\label{al1}
\end{lemma}

\begin{proof}
If $k=2m$ then $\alpha_k(x,y)=1-(y+2-x^2)S_{m-1}(y) \left( S_{m-1}(y) - S_{m-2}(y) \right)$ and $\lambda=2+(y-2)(y+2-x^2)S^2_{m-1}(y)$ by Lemmas \ref{wabw} and \ref{trace}. Hence, by direct calculations using $x=2$, we have
\begin{eqnarray*}
\alpha^2_{k}(x,y)-\alpha_{k}(x,y)\lambda+1 &=& \left(-1-S^2_{m-1}(y)+yS^2_{m-1}(y)-yS_{m-1}(y)S_{m-2}(y)+S^2_{m-2}(y) \right)  \\
&&  \times (y-2)^2 S^2_{m-1}(y).
\end{eqnarray*}
Since $S^2_{m-1}(y)-yS_{m-1}(y)S_{m-2}(y)+S^2_{m-2}(y)=1$ (by Lemma \ref{S}), we obtain
$$\alpha^2_{k}(x,y)-\alpha_{k}(x,y)\lambda+1=(y-2)^3S^4_{m-1}(y).$$

If $k=2m+1$ then $\alpha_k(x,y)=1+(y+2-x^2)S_{m-1}(y) \left( S_m(y) - S_{m-1}(y) \right)$ and $\lambda=x^2-y-(y-2)(y+2-x^2)S_m(y)S_{m-1}(y)$ by Lemmas \ref{wabw} and \ref{trace}. Hence, by direct calculations using $x=2$, we have
\begin{eqnarray*}
\alpha^2_{k}(x,y)-\alpha_{k}(x,y)\lambda+1 &=& (y-2) \big( 1+ (2-y)S^2_{m-1}(y)+(y-2)S^4_{m-1}(y)\\
&& + \, 2 (y-2)S_{m-1}(y)S_m(y)+(2-y)yS^3_{m-1}(y)S_m(y)\\
&& + \, (y^2-3y+2)S^2_{m-1}(y)S^2_m(y)\big)
\end{eqnarray*}
By replacing $yS^3_{m-1}(y)S_m(y)=S^2_{m-1}(y) \left( S^2_{m-1}(y)+S^2_{m}(y)-1\right)$ in the above equality, 
$$\alpha^2_{k}(x,y)-\alpha_{k}(x,y)\lambda+1=(y-2) \left( (y-2)S_m(y)S_{m-1}(y)+1 \right)^2$$
as claimed.
\end{proof}

\begin{proposition}
Suppose $y$ is a real number satisfying the equation $\phi_{k,2n}(2,y)=0$. Then $y>2$.
\label{al}
\end{proposition}

\begin{proof}
Suppose $\phi_{k,2n}(x,y)=0$. Then $ S_{n-1}(\lambda)\alpha_k(x,y)=S_{n-2}(\lambda)$. Hence
\begin{eqnarray*}
1 &=& S^2_{n-1}(\lambda)-\lambda S_{n-1}(\lambda)S_{n-2}(\lambda)+S^2_{n-2}(\lambda)\\
&=& \left( \alpha^2_{k}(x,y)-\alpha_{k}(x,y)\lambda+1 \right) S^2_{n-1}(\lambda).
\end{eqnarray*}
If we also suppose that $x=2$ and $y$ is a real number, then the above equality implies that $\alpha^2_{k}(x,y)-\alpha_{k}(x,y)\lambda+1>0$. By Lemma \ref{al1}, we must have $y>2$.
\end{proof}

\section{Twisted Alexander polynomials}

In this section we explicitly calculate the coefficients 
of the highest and lowest degree terms of the twisted Alexander polynomial associated to a non-abelian representation of $J(k,2n)$ and give the proof 
of Theorem~\ref{thm:main-theorem} (1).

\subsection{Twisted Alexander polynomials}

For a knot group $G_K=\pi_1(S^3\backslash K)$, 
we choose and fix a Wirtinger presentation 
$$
G_K=
\langle a_1,\ldots,a_q~|~r_1,\ldots,r_{q-1}\rangle.
$$
Then 
the abelianization homomorphism 
$f:G_K\to H_1(S^3\backslash K,\BZ)
\cong {\BZ}
=\langle t
\rangle$ 
is given by 
$f(a_1)=\cdots=f(a_q)=t$. 
Here 
we specify a generator $t$ of $H_1(S^3\backslash K;\BZ)$ 
and denote the sum in $\BZ$ multiplicatively. 
Let us 
consider a linear representation 
$\rho:G_K\to SL_2(\BC)$. 

The maps $\rho$ and $f$ naturally induce two ring homomorphisms 
$\widetilde{\rho}: {\BZ}[G_K] \rightarrow M(2,{\BC})$ 
and $\widetilde{f}:{\BZ}[G_K]\rightarrow {\BZ}[t^{\pm1}]$ respectively, 
where ${\BZ}[G_K]$ is the group ring of $G_K$ 
and 
$M(2,{\BC})$ is the matrix algebra of degree $2$ over ${\BC}$. 
Then 
$\widetilde{\rho}\otimes\widetilde{f}$ 
defines a ring homomorphism 
${\BZ}[G_K]\to M\left(2,{\BC}[t^{\pm1}]\right)$. 
Let 
$F_q$ denote the free group on 
generators $a_1,\ldots,a_q$ and 
$\Phi:{\BZ}[F_q]\to M\left(2,{\BC}[t^{\pm1}]\right)$
the composition of the surjection 
${\BZ}[F_q]\to{\BZ}[G_K]$ 
induced by the presentation of $G_K$ 
and the map 
$\widetilde{\rho}\otimes\widetilde{f}:{\BZ}[G_K]\to M(2,{\BC}[t^{\pm1}])$. 

Let us consider the $(q-1)\times q$ matrix $M$ 
whose $(i,j)$-component is the $2\times 2$ matrix 
$$
\Phi\left(\frac{\partial r_i}{\partial a_j}\right)
\in M\left(2,{\BZ}[t^{\pm1}]\right),
$$
where 
${\partial}/{\partial a}$ 
denotes the free differential calculus. 
For 
$1\leq j\leq q$, 
let us denote by $M_j$ 
the $(q-1)\times(q-1)$ matrix obtained from $M$ 
by removing the $j$th column. 
We regard $M_j$ as 
a $2(q-1)\times 2(q-1)$ matrix with coefficients in 
${\BC}[t^{\pm1}]$. 
Then Wada's twisted Alexander polynomial 
of a knot $K$ 
associated to a representation $\rho:G_K\to SL_2({\BC})$ 
is defined to be a rational function 
$$
\Delta_{K,\rho}(t)
=\frac{\det M_j}{\det\Phi(1-a_j)} 
$$
and 
moreover well-defined 
up to a factor $t^{2n}~(n\in{\BZ})$. 
It is also known that if two representations $\rho,\rho'$ are conjugate, 
then $\Delta_{K,\rho}(t)=\Delta_{K,\rho'}(t)$ holds. 
See \cite{Wada94-1} and \cite{GKM05-1} for details. 

\begin{remark}
Let $\rho:G_K\to SL_2(\BC)$ be a non-abelian representation. 
\begin{enumerate}
\item
The twisted Alexander polynomial $\Delta_{K,\rho}(t)$ 
associated to $\rho$ is always a Laurent polynomial 
for any knot $K$ \cite{KM05-1}. 
\item
The twisted Alexander polynomial is reciprocal, 
i.e. $\Delta_{K,\rho}(t)=t^i\Delta_{K,\rho}(t^{-1})$ holds for
some $i\in \BZ$ (see \cite{HSW10-1} and \cite{FKK11-1}).
\item
If $K$ is a fibered knot, then $\Delta_{K,\rho}(t)$ is 
a monic polynomial for every non-abelian representation $\rho$ \cite{GKM05-1}. 
It is also known that the converse holds for alternating knots \cite[Remark~4.2]{KM}. 
\item
If $K$ is a knot of genus $g$, then $\deg(\Delta_{K,\rho}(t))\leq 4g-2$ \cite{FK06-1}. 
Moreover if $K$ is fibered, then the equality holds \cite{KM05-1}. 
\end{enumerate}
\label{genus}
\end{remark}

We say the twisted Alexander polynomial $\Delta_{K,\rho}(t)$ determines the knot genus 
$g(K)$ if $\deg(\Delta_{K,\rho}(t))=4g(K)-2$ holds. For a hyperbolic knot $K$, 
the hyperbolic torsion polynomial $\CT_K(t)$ is defined to be $\Delta_{K,\rho_0}(t)$ 
for the holonomy representation $\rho_0:G_K\to SL_2(\BC)$. 
We note that it is normalized so that $\CT_K(t)=\CT_K(t^{-1})$ holds.

\subsection{Proof of Theorem~\ref{thm:main-theorem} (1)} It is known that the genus of $J(k,2n)$, where $k>1$ and $|n| > 0$,  is $1$ if $k$ is even, and is $|n|$ if $k$ is odd. Moreover, the genus of $J(1,2n)$ (the $(2,2n-1)$-torus knot) is $n-1$ if $n>0$ and is $-n$ if $n<0$. 

We first consider the case $n>0$. Let $r=w^naw^{-n}b^{-1}$, where $w$ is as defined in Section \ref{nonab}. By direct calculations, we have 
\begin{equation}
\frac{\partial r}{\partial a} = w^n \left( 1+(1-a)(w^{-1}+\dots+w^{-n}) \frac{\partial w}{\partial a}\right)
\label{M2}
\end{equation}
where $\frac{\partial w}{\partial a} =$
$$ 
\begin{cases}
 -\left( ba^{-1}+\dots+(ba^{-1})^m \right)+(ba^{-1})^m \left( 1+b^{-1}a+\dots+(b^{-1}a)^{m-1})b^{-1} \right), & k=2m,\\
 -\left( ba^{-1}+\dots+(ba^{-1})^m \right)+(ba^{-1})^m b\left( 1+ab^{-1}+\dots+(ab^{-1})^{m}) \right), & k=2m+1.
 \end{cases}
 $$

Suppose $\rho: G_K \to SL_2(\BC)$ is a non-abelian representation given by \eqref{nonabelian}. Then the twisted Alexander polynomial of $K$ associated to $\rho$ is 
$$\Delta_{K,\rho}(t)=\det \Phi \left( \frac{\partial r}{\partial a} \right) \big/\det \Phi(1-b)=\det \Phi \left( \frac{\partial r}{\partial a} \right) \big/ (1-tx+t^2).$$


\subsubsection{\textbf{The case} $J(2m,2n), \, n>0$}\label{subsection:(2m,2n)n>0} 
From \eqref{M2} we have
$$
\det \Phi \left( \frac{\partial r}{\partial a} \right)=|I+(I-tA)(W^{-1}+\dots+W^{-n})V|
$$
where $I$ is the $2 \times 2$ identity matrix and
$$V=-(BA^{-1}+\dots+(BA^{-1})^m)+(BA^{-1})^m(I+B^{-1}A+\dots+(B^{-1}A)^{m-1})t^{-1}B^{-1}.$$
The following lemma follows easily.

\begin{lemma}
\begin{enumerate}
\item
The highest degree term of $\det \Phi \left( \frac{\partial r}{\partial a} \right)$ is 
$$|A(W^{-1}+\dots+W^{-n})( BA^{-1}+\dots+(BA^{-1})^m )|t^2.$$
\item
The lowest degree term of $\det \Phi \left( \frac{\partial r}{\partial a} \right)$ is
$$|(W^{-1}+\dots+W^{-n})(BA^{-1})^m(I+B^{-1}A+\dots+(B^{-1}A)^{m-1})B^{-1}|t^{-2}.$$
\end{enumerate}
\label{2m,2n}
\end{lemma}

 Let $\{T_i(z)\}_i$ be the sequence of Chebyshev polynomials defined by $T_0(z)=2,\,T_1(z)=z$ and $T_{i+1}(z)=zT_i(z)-T_{i-1}(z)$ for all integers $i$. Recall that $y=\tr AB^{-1}$ and $\lambda = \tr W$.

\begin{proposition}
The highest and lowest degree terms of $\det \Phi \left( \frac{\partial r}{\partial a} \right)$ are $\frac{T_n(\lambda)-2}{\lambda-2} \times  \frac{T_m(y)-2}{y-2} \, t^2$  and $\frac{T_n(\lambda)-2}{\lambda-2} \times  \frac{T_m(y)-2}{y-2} \, t^{-2}$ respectively. 
\label{prop1}
\end{proposition}

\begin{proof}
Let $\beta_{\pm}$ (resp. $\gamma_{\pm}$ ) are roots of $z^2-\lambda z+1$ (resp. $z^2-yz+1$). Lemma \ref{2m,2n} implies that the highest and lowest degree terms  of $\det \Phi \left( \frac{\partial r}{\partial a} \right)$ are
\begin{eqnarray*}
(1+\beta_+ +\dots+\beta_+^{n-1})(1+\beta_- +\dots+\beta_-^{n-1})(1+\gamma_+ +\dots+\gamma_+^{m-1})(1+\gamma_- +\dots+\gamma_-^{m-1})t^2
\end{eqnarray*}
and 
\begin{eqnarray*}
(1+\beta_+ +\dots+\beta_+^{n-1})(1+\beta_- +\dots+\beta_-^{n-1})(1+\gamma_+ +\dots+\gamma_+^{m-1})(1+\gamma_- +\dots+\gamma_-^{m-1})t^{-2}
\end{eqnarray*}
respectively. Proposition \ref{prop1} then follows from Lemma \ref{simple} below.
\end{proof}

\begin{lemma}
One has $$(1+\beta_+ +\dots+\beta_+^{n-1})(1+\beta_- +\dots+\beta_-^{n-1})=\frac{T_n(\lambda)-2}{\lambda-2} \in \BZ[\lambda].$$
\label{simple}
\end{lemma}

\begin{proof}
We have
$$LHS=\frac{(\beta_+^n-1)(\beta_-^n-1)}{(\beta_+-1)(\beta_--1)}=\frac{\beta_+^n + \beta_-^n -2}{\beta_+ + \beta_- -2 }=\frac{T_n(\lambda)-2}{\lambda-2}.$$
The lemma follows.
\end{proof}

Proposition \ref{prop1} implies that the highest and lowest degree terms of 
the twisted Alexander polynomial $\Delta_{K,\rho}(t)=\det \Phi \left( \frac{\partial r}{\partial a} \right) \big/ (1-tx+t^2)$ are $\frac{T_n(\lambda)-2}{\lambda-2} \times  \frac{T_m(y)-2}{y-2} \, t^0$  and $\frac{T_n(\lambda)-2}{\lambda-2} \times  \frac{T_m(y)-2}{y-2} \, t^{-2}$ respectively. 
Hence to prove Theorem \ref{thm:main-theorem} (1) for $J(2m,2n),~n > 0$, 
we only need to show that the coefficients of these terms are nonzero 
under the assumption that $\phi_K(2,y)= S_{n-1}(\lambda)\alpha_{k}(2,y)-S_{n-2}(\lambda)=0$ 
(because the roots of this equation correspond to the parabolic representations). 
To this end, 
we show that at $x=2$ the polynomials $\phi_K(2,y)$ and $\frac{T_n(\lambda)-2}{\lambda-2} \times  \frac{T_m(y)-2}{y-2}$ 
do not have any common zero $y \in \BC$ 
(in fact, if they have a common zero, the highest and lowest degree terms vanish at $x=2$). 
It is equivalent to show that at $x=2$ these polynomials are relatively prime in $\BC[y]$. 


Recall that $\lambda=\tr W=(y-2)(y+2-x^2)S^2_{m-1}(y)+2$ and
$$
\alpha_k(x,y) = 1-(y+2-x^2)S_{m-1}(y) \left( S_{m-1}(y) - S_{m-2}(y) \right).
$$
Lemmas \ref{twist} and \ref{2} below will complete the proof of Theorem \ref{thm:main-theorem} (1) for $J(2m,2n)$, $n > 0$.

\begin{lemma}
Suppose $x=2$. Then $\gcd(\phi_K(x,y),\frac{T_n(\lambda)-2}{\lambda-2})=1$ in $\BC[y]$.
\label{twist}
\end{lemma}

\begin{proof} It is equivalent to show that at $x=2$, $\phi_K(x,y)$ and $\frac{T_n(\lambda)-2}{\lambda-2}$ do not have any common root $y \in \BC$.

Suppose $T_n(\lambda)=2$ and $\lambda \not= 2$ then $\beta_+^n = \beta_-^n=1$ and $\beta_{+} \not= 1$. If $\beta_{+} \not= -1$ then $S_{n-1}(\lambda)=\frac{\beta_+^n - \beta_-^n}{\beta_+ - \beta_-}=0$ and $S_{n-2}(\lambda)=\frac{\beta_+^{n-1} - \beta_-^{n-1}}{\beta_+ - \beta_-}=-1$; hence $\phi_K(x,y)= 1 \not= 0$.

If $\beta_{+}=-1$ (in this case $n$ must be even) then $\lambda=-2$. It implies that $S_{n-1}(\lambda)=-n$, and $S_{n-2}(\lambda)=n-1$. Hence $\phi_K(x,y)=-n \, \alpha_k(x,y)-(n-1)$.

Suppose $\phi_K(x,y)=0$. Then $\alpha_k(x,y)=\frac{1}{n}-1$. We have
\begin{eqnarray*}
(y-2)(y+2-x^2)S^2_{m-1}(y) &=& \lambda-2=-4,\\
(y+2-x^2) \left( S^2_{m-1}(y)-S_{m-1}(y)S_{m-2}(y) \right) &=& 1-\alpha_k(x,y)=2-\frac{1}{n}.
\end{eqnarray*}
It implies that $(y+2-x^2)S^2_{m-1}(y)=\frac{-4}{y-2}$ and $(y+2-x^2)S_{m-1}(y)S_{m-2}(y)=-\left( \frac{4}{y-2}+2-\frac{1}{n} \right)$.

Since $$S^4_{m-1}(y)-yS^2_{m-1}(y)(S_{m-1}(y)S_{m-2}(y)) + (S_{m-1}(y)S_{m-2}(y))^2=S^2_{m-1}(y)$$
(by Lemma \ref{S}), we must have 
$$\frac{16}{(y-2)^2}-\frac{4y}{y-2} \left( \frac{4}{y-2}+2-\frac{1}{n} \right) + \left( \frac{4}{y-2}+2-\frac{1}{n} \right)^2=\frac{-4(y+2-x^2)}{y-2}$$
i.e. $y=2+4n^2x^2$.

If $x \in \BR$ then $y=2+4n^2x^2 \in \BR$ and $$-4=(y-2)(y+2-x^2)S_{m-1}^2(y)=4n^2x^2(4+(4n^2-1)x^2)S^2_{m-1}(y) \ge 0,$$ a contradiction. Hence $\phi_K(x,y) \not= 0$ when $\frac{T_n(\lambda)-2}{\lambda-2}=0$ and $x \in \BR$. The lemma follows.
\end{proof}
 
\begin{lemma}
Suppose $x=2$. Then $\gcd(\phi_K(x,y),\frac{T_m(y)-2}{y-2})=1$ in $\BC[y]$.
\label{2}
\end{lemma}

\begin{proof}
Suppose $T_m(y)=2$ and $y \not= 2$ then $\gamma_+^m =\gamma_-^m=1$ and $\gamma_{+} \not= 1$. If $\gamma_{+} \not= -1$ then $S_{m-1}(y)=\frac{\gamma_+^m - \gamma_-^m}{\gamma_+ - \gamma_-}=0$, hence $\lambda=2$ and $\alpha_k(x,y)=1$. It implies that $\phi_K(x,y)=S_{n-1}(2)-S_{n-2}(2)=1 \not= 0$.

If $\gamma_{+} = -1$ (in this case $m$ must be even) then $y=-2$. We have $\lambda=(y-2)(y+2-x^2)S^2_{m-1}(y)+2=4m^2x^2+2$ and $$\alpha_k(x,y)=1-(y+2-x^2)S_{m-1}(y)(S_{m-1}(y)-S_{m-2}(y))=m(2m-1)x^2+1.$$ It implies that
$$\phi_K(x,y)=S_{n-1}(\lambda)\alpha_k(x,y)-S_{n-1}(\lambda)=\left( m(2m-1)x^2+1 \right) S_{n-1}(\lambda)-S_{n-1}(\lambda).$$

If $x \in \BZ$ then $\lambda=4m^2x^2+2 \in \BZ$ is even. Hence $\phi_K(x,y) \equiv S_{n-1}(\lambda)-S_{n-2}(\lambda) \pmod{2}$ is odd and so is non-zero. Therefore $\phi_K(x,y) \not=0$ when $\frac{T_m(y)-2}{y-2}=0$ and $x \in \BZ$. The lemma follows.
\end{proof}

\subsubsection{\textbf{The case} $J(2m+1,2n), \, n>0$}\label{subsection:(2m+1,2n)n>0}

From \eqref{M2} we have
$$\det \Phi \left( \frac{\partial r}{\partial a} \right) =  |I+(I-tA)(t^{-2}W^{-1}+ \dots + t^{-2n}W^{-n})V|t^{4n}$$
where
$$V=-(BA^{-1}+\dots+(BA^{-1})^m)+t(BA^{-1})^mB(I+AB^{-1}+\dots+(AB^{-1})^{m}).$$

We first consider the case $m=0$ (in this case we must have $n>1$ so that $K$ is a non-trivial knot). Then $W=BA$ and
\begin{eqnarray*}
\det \Phi \left( \frac{\partial r}{\partial a} \right) &=&  |I+(I-tA)(t^{-2}W^{-1}+ \dots + t^{-2n}W^{-n})tB| \, t^{4n}\\
&=& |(t^{-2}W^{-1}+ \dots + t^{-2n}W^{-n})tB-tA(t^{-4}W^{-2}+ \dots + t^{-2n}W^{-n})tB| \, t^{4n}.
\end{eqnarray*}
It implies that the highest and lowest degree terms of $\det \Phi \left( \frac{\partial r}{\partial a} \right)$ are $|t^{-1}W^{-1}B|t^{4n}=t^{4n-2}$ and $|t^{1-2n}W^{-n}B|t^{4n}=t^2$ respectively. Hence the highest and lowest degree terms of $\Delta_{K,\rho}(t)$ are $t^{4n-4}$ and 
$t^2$ respectively. 
Since the genus of $J(1,2n)$, where $n>1$, is $n-1$, we complete the proof of Theorem \ref{thm:main-theorem} (1) for $J(1,2n), \,n > 1$.

We now consider the case $m >0$. In this case, we have the following.

\begin{lemma} 
\begin{enumerate}
\item
The highest degree term of $\det \Phi \left( \frac{\partial r}{\partial a} \right)$ is 
$$|I-AW^{-1}(BA^{-1})^mB(I+AB^{-1}+\dots+(AB^{-1})^m)|t^{4n}$$
$$=|I+BA^{-1}+\dots+(BA^{-1})^{m-1}|t^{4n}.$$

\item
The lowest degree term of $\det \Phi \left( \frac{\partial r}{\partial a} \right)$ is
$$
|-W^{-n}(BA^{-1}+\dots+(BA^{-1})^m)|t^{0}=|I+BA^{-1}+\dots+(BA^{-1})^{m-1}|t^0.$$
\end{enumerate}
\label{2m+1,2n}
\end{lemma}

Lemmas \ref{2m+1,2n} and \ref{simple} imply the following.

\begin{proposition}
The highest and lowest degree terms of $\det \Phi \left( \frac{\partial r}{\partial a} \right)$ are
$
\frac{T_m(y)-2}{y-2} \, t^{4n}
$ and $\frac{T_m(y)-2}{y-2} \, t^{0}$ respectively.
\label{prop2}
\end{proposition}

Proposition \ref{prop2} implies that the highest and lowest degree terms of $\Delta_{K,\rho}(t)$ are $\frac{T_m(y)-2}{y-2} \, t^{4n-2}$  and $\frac{T_m(y)-2}{y-2} \, t^0$ respectively. Hence to prove Theorem \ref{thm:main-theorem} (1) for $J(2m+1,2n)$, where $m, n > 0$, we only need to show that at $x=2$ (parabolic representation) the polynomials $\phi_K(x,y)= S_{n-1}(\lambda)\alpha_{k}(x,y)-S_{n-2}(\lambda)$ and $\frac{T_m(y)-2}{y-2}$ are relatively prime in $\BC[y]$. 

Recall that $\lambda=\tr w=x^2-y-(y-2)(y+2-x^2)S_{m}(y)S_{m-1}(y)$ and
$$\alpha_k(x,y) =1+(y+2-x^2)S_{m-1}(y) \left( S_m(y) - S_{m-1}(y) \right) .$$
Lemma \ref{lem1} below will complete the proof of Theorem \ref{thm:main-theorem} (1) for $J(2m+1,2n)$, where $m, n > 0$.
 
\begin{lemma}
Suppose $x=2$. Then $\gcd(\phi_K(x,y),\frac{T_m(y)-2}{y-2})=1$ in $\BC[y]$.
\label{lem1}
\end{lemma}

\begin{proof}
Suppose $T_m(y)=2$ and $y \not= 2$, then $\gamma_+^m = \gamma_-^m=1$ and $\gamma_{+} \not= 1$. If $\gamma_{+} \not= -1$ then $S_{m-1}(y)=0$ and $S_{m}(y)=1$, hence $\lambda=x^2-y$ and $\alpha_k(x,y)=1$. It implies that $\phi_K(x,y)=S_{n-1}(\lambda)-S_{n-2}(\lambda).$ 

Since $\gamma_+^m=1$, we have $y=\gamma_+ + \gamma_+^{-1}=2\cos \frac{2\pi j}{m}$ for some $0<j<m$. If $\phi_K(x,y)=S_{n-1}(\lambda)-S_{n-2}(\lambda)=0$ then $\lambda=2\cos \frac{(2j'-1)\pi}{2n-1}$ for some $1 \le j' \le n-1$, see e.g. \cite[Lemma 4.13]{LT}. Hence $x^2=y+\lambda=2(\cos \frac{2\pi j}{m}+\cos \frac{(2j'-1)\pi}{2n-1}) < 4.$

If $\gamma_{+} = -1$ (in this case $m$ must be even) then $y=-2$. We have $\lambda=-(y-2)(y+2-x^2)S_{m}(y)S_{m-1}(y)+x^2-y=(2m+1)^2x^2+2$ and 
$$\alpha_k(x,y)=1+(y+2-x^2)S_{m-1}(y) \left( S_m(y) - S_{m-1}(y) \right)=m(2m+1)x^2+1.$$
If $x$ is an even integer then $\lambda=(2m+1)^2x^2+2$ is an even integer and $\alpha_k(x,y)=m(2m+1)x^2+1$ is an odd integer. Hence 
$$\phi_K(x,y)=S_{n-1}(\lambda)\alpha_k(x,y)-S_{n-2}(\lambda) \equiv S_{n-1}(\lambda)-S_{n-2}(\lambda) \pmod{2}$$ is odd and so is non-zero. 

In both cases, we obtain $\phi_K(x,y) \not= 0$ when $\frac{T_m(y)-2}{y-2}=0$ and $x$ is an even integer $\ge 2$. The lemma follows.
\end{proof}

Next we consider the case $n<0$. 
We put $l=-n~(l>0)$. 
For $r=w^naw^{-n}b^{-1}=w^{-l}aw^lb^{-1}$, 
we have 
\begin{align}
\frac{\p r}{\p a}
&=
\frac{\p w^{-l}}{\p a}+w^{-l}\left(1+a\frac{\p w^l}{\p a}\right) \nonumber\\
&=
w^{-l}
\left(
1-(1-a)\left(1+w+\cdots+w^{l-1}\right)\frac{\p w}{\p a}
\right).
\label{n<0}
\end{align}

\subsubsection{\textbf{The case} $J(2m,2n), \, n<0$}
From \eqref{n<0} we have 
$$
\det \Phi \left( \frac{\partial r}{\partial a} \right)
=
\left|
I-(I-tA)(I+W+\cdots+W^{l-1})V
\right|,
$$
where 
$$
V=-(BA^{-1}+\dots+(BA^{-1})^m)+(BA^{-1})^m(I+B^{-1}A+\dots+(B^{-1}A)^{m-1})t^{-1}B^{-1}.
$$

\begin{lemma}
\begin{enumerate}
\item
The highest degree term of $\det \Phi \left( \frac{\partial r}{\partial a} \right)$ is 
$$
\left|
-A(I+W+\cdots+W^{l-1})(BA^{-1}+\cdots+(BA^{-1})^m)
\right|t^2.
$$

\item
The lowest degree term of $\det \Phi \left( \frac{\partial r}{\partial a} \right)$ is 
$$
\left|
-(I+W+\cdots+W^{l-1})(BA^{-1})^m(I+B^{-1}A+\cdots+(B^{-1}A)^{m-1})B^{-1}
\right|t^{-2}.
$$
\end{enumerate}
\end{lemma}

We can apply the similar argument as in Subsection~\ref{subsection:(2m,2n)n>0} and 
hence $\Delta_{K,\rho}(t)$, where $\rho$ is parabolic, determines the knot genus in this case. 

\subsubsection{\textbf{The case} $J(2m+1,2n), \, n<0$}\label{subsection:(2m+1,2n)n<0}
From \eqref{n<0} we have 
\begin{align*}
\det \Phi \left( \frac{\partial r}{\partial a} \right)
&=
\left|
t^{-2l}W^{-l}
\left(
I-(I-tA)(I+t^2W+\cdots+t^{2(l-1)}W^{l-1})V
\right)
\right|\\
&=
\left|
I-(I-tA)(I+t^2W+\cdots+t^{2(l-1)}W^{l-1})V
\right| \, t^{-4l},
\end{align*}
where
$$
V=-(BA^{-1}+\dots+(BA^{-1})^m)+t(BA^{-1})^mB(I+AB^{-1}+\dots+(AB^{-1})^{m}).
$$

\begin{lemma}
\begin{enumerate}
\item
The highest degree term of $\det \Phi \left( \frac{\partial r}{\partial a} \right)$ is 
$$
\left|
AW^{l-1}(BA^{-1})^mB(I+AB^{-1}+\cdots+(AB^{-1})^m)
\right|t^0
=
|I+AB^{-1}+\cdots+(AB^{-1})^m|t^0.
$$

\item
The lowest degree term of $\det \Phi \left( \frac{\partial r}{\partial a} \right)$ is 
$$
\left|
I+BA^{-1}+\cdots+(BA^{-1})^m
\right|t^{-4l}.
$$
\end{enumerate}
\label{prop3}
\end{lemma}

We can again apply the similar argument as in Subsection~\ref{subsection:(2m+1,2n)n>0} 
and hence $\Delta_{K,\rho}(t)$, where $\rho$ is parabolic, determines the knot genus in this case. 

By Subsections~\ref{subsection:(2m,2n)n>0} to \ref{subsection:(2m+1,2n)n<0}, 
Theorem~\ref{thm:main-theorem} (1) immediately follows. 

\section{The fibering problem}
\label{condition}

In this section we study some properties of the parabolic representation spaces of 2-bridge knots and give the proof of Theorem \ref{thm:main-theorem} (2).

\subsection{Parabolic representations of 2-bridge knots} Consider the 2-bridge knot $K=\fb(p,q)$, 
where $p>q \ge 1$ are relatively prime. 
The knot group $G_K$ has a presentation $G_K=\langle a,b \mid wa=bw \rangle$, 
where $w=a^{\varepsilon_1}b^{\varepsilon_2} \cdots a^{\varepsilon_{p-2}}b^{\varepsilon_{p-1}}$ 
and $\varepsilon_j=(-1)^{\lfloor jq/p\rfloor}$, see e.g. \cite{BZ}.

Let $\phi_K(x,y)$ be the defining equation for the non-abelian representations into $SL_2(\BC)$ of 
$G_K$, 
where 
$x=\tr \rho(a)=\tr \rho(b)$ and $y=\tr \rho(ab^{-1})$. 
Then $\phi_K(2,y)$ is the defining equation for the parabolic representations. It is known that $\phi_K(2,y) \in \BZ[y]$ is a monic polynomial of degree $d=\frac{p-1}{2}$, see \cite{Ri, Le93}. 

We want to study the irreducibility of $\phi_K(2,y) \in \BZ[y]$. 

\begin{lemma}
One has $\phi_{K}(2,y) = S_d(y)+S_{d-1}(y)$ in $\BZ_2[y]$.
\label{mod2}
\end{lemma}

\begin{proof}
The proof is similar to that of \cite[Proposition A.2]{LT}. 

Suppose $\rho$ is a parabolic representation. Let $A=\rho(a), \,B=\rho(b)$ and $W=\rho(w)$.  Taking conjugation if necessary, we can assume that 
\begin{equation}
A=\left[ \begin{array}{cc}
1 & 1\\
0 & 1 \end{array} \right] \quad \text{and} \quad B=\left[ \begin{array}{cc}
1 & 0\\
2-y & 1 \end{array} \right]
\end{equation}
where $y=\tr AB^{-1} \in \BC$ satisfies the matrix equation $WA-BW=0$. 

By the Cayley-Hamilton theorem applying for matrices in $SL_2(\BC)$ we have $A+A^{-1}=\tr(A)I=2I = 0 \pmod{2}$, i.e. $A^{-1} = A \pmod{2}$. Similarly, $B^{-1}=B \pmod{2}$. It implies that $W=A^{\ve_1}B^{\ve_2} \dots A^{\ve_{2d-1}}B^{\ve_{2d}}=(AB)^d \pmod{2}$. By applying \eqref{-1}, we have
\begin{eqnarray*}
WA+BW &=& (AB)^d A +B(AB)^d \\
&=& 
S_{d}(y)(A+B)+S_{d-1}(y)(B^{-1}+A^{-1})\\
&=&  \left( S_d(y)+S_{d-1}(y) \right) (A+B) \pmod{2}
\end{eqnarray*}
where $A+B=\left[ \begin{array}{cc}
0 & 1\\
y & 0 \end{array} \right] \pmod 2$. Hence $\phi_K(2,y)=S_d(y)+S_{d-1}(y)$ in $\BZ_2[y]$.
\end{proof}

Recall from the Introduction that $\CP_2$ is the set of all odd primes $p$ such that $2$ is a primitive root modulo $p$.

\begin{lemma}
Suppose $p \in \CP_2$. Then $S_d(y)+S_{d-1}(y) \in \BZ_2[y]$ is irreducible.
\label{irredmod2}
\end{lemma}

\begin{proof}
Let $y=u+u^{-1}$. Then $$S_d(y)+S_{d-1}(y)=\frac{u^{d+1}+u^{-(d+1)}}{u+u^{-1}}+\frac{u^{d}+u^{-d}}{u+u^{-1}}=u^{-d}\,\frac{1+u^{2d+1}}{1+u}.$$ 

Suppose $p \in \CP_2$. We will show that $\frac{1+u^{p}}{1+u} \in \BZ_2[u]$ is irreducible. This will imply that $S_d(y)+S_{d-1}(y) \in \BZ_2[y]$ is irreducible.

We have $\frac{1+u^{p}}{1+u}=u^{p-1}+\dots+u+1$ is the $p^{th}$-cyclotomic polynomial $C_p(u) \in \BZ_2[u]$ (since $p$ is an odd prime). It is well known that $C_p(u) \in \BZ_2[u]$ is irreducible if $p \in \CP_2$, see e.g. \cite[Theorem 11.2.8]{Ro}. The lemma follows.
\end{proof}

\begin{proposition}
Suppose $p \in \CP_2$. Then $\phi_{K}(2,y) \in \BZ[y]$ is irreducible.
\label{Zirred}
\end{proposition}

\begin{proof}
By Lemma \ref{mod2}, $\phi_{K}(2,y) = S_d(y)+S_{d-1}(y) \in \BZ_2[y]$. Since $p \in \CP_2$, $S_d(y)+S_{d-1}(y) \in \BZ_2[y]$ is irreducible by Lemma \ref{irredmod2}. It implies that $\phi_{K}(2,y)$ is irreducible in $\BZ_2[y]$. Since $\phi_K(2,y) \in \BZ[y]$ is a monic polynomial in $y$, it is irreducible in $\BZ[y]$.
\end{proof}

\subsection{Proof of Theorem~\ref{thm:main-theorem} (2)} It is known that $J(k,2n)$ is fibered only for the trivial knot $J(k,0)$, the trefoil knot $J(2,-2)$, the figure eight knot $J(2,2)$, the knots $J(1,2n)$ for any $n$, and the knots $J(3,2n)$ for $n>0$.

\subsubsection{\textbf{The case} $J(2m,2n),\, m>1$} \label{subsection:(2m,2n)} 

We will apply the result in the previous subsection to study the fibering problem for $K=J(2m,2n)$. 

Let $p=|4mn-1|$ then it is known that $\phi_K(2,y)$ has degree $\frac{p-1}{2}$. By Proposition \ref{Zirred}, the polynomial $\phi_K(2,y) \in \BZ[y]$ is irreducible 
if $p \in \CP_2$.

\begin{proposition}
Suppose $m>1$ and $p=|4mn-1| \in \CP_2$. Then $\Delta_{K,\rho}(t)$ is non-monic for every parabolic representation $\rho$.
\end{proposition}

\begin{proof}
We only need to consider the case $n>0$. The case $n<0$ is similar.

Suppose $\rho$ is a parabolic representation, i.e. $x=2$. Since $k=2m$, by Proposition \ref{prop1} the coefficient of the highest degree term of $\Delta_{K,\rho}(t)$ is $h(y)=\frac{T_n(\lambda)-2}{\lambda-2} \times  \frac{T_m(y)-2}{y-2}$, an integer polynomial in $y$ of degree $(n-1)(2m)+(m-1)=2mn-(m+1)<2mn-1=\frac{p-1}{2}$.

Since $p \in \CP_2$, $\phi_K(2,y) \in \BZ[y]$ is irreducible. It implies that $\phi_K(2,y)$ does not divide $h(y)-1$ in $\BZ[y]$. Hence $h(y) \not= 1$ when $\phi_K(2,y)=0$. The proposition follows. 
\end{proof}

\subsubsection{\textbf{Twist knots} $J(2,2n)$}\label{subsection:(2,2n)} For $K=J(2,2n)$ we have $\lambda=y^2-yx^2+2x^2-2$ and $\phi_K(x,y)=-(y+1-x^2)S_{n-1}(\lambda)-S_{n-2}(\lambda)$. Suppose $\rho$ is a non-abelian representation. By Proposition \ref{prop1} the coefficient of the highest degree term of $\Delta_{K,\rho}(t)$ is $\frac{T_n(\lambda)-2}{\lambda-2}$. We want to show that for $|n|>1$, we have $\frac{T_n(\lambda)-2}{\lambda-2} \not= 1$ when $\phi_K(x,y)=0$ and $x=2$. This will imply that for any parabolic representation $\rho$, $\Delta_{K,\rho}(t)$ is monic if and only if $|n|=1$.

\begin{lemma}\label{lem:twist-knot}
Suppose $x=2$. For $|n|>1$, one has $\gcd(\phi_K(x,y),\frac{T_n(\lambda)-2}{\lambda-2}-1)=1$ in $\BC[y]$.
\end{lemma}

\begin{proof}
We only need to consider the case $n>1$. The case $n<-1$ is similar.

Suppose $T_n(\lambda)=\lambda$ and $\lambda \not= 2$. Then $\beta_+^n + \beta_-^n=\beta_+ + \beta_-$, i.e., $\beta_+^{n-1}=1$ or $\beta_+^{n+1}=1$. It implies that $\lambda=-2$, or $\lambda=2\cos \frac{2j\pi}{n-1}$ for some $1 \le j \le n-2$ and $j \not=  \frac{n-1}{2}$, or $\lambda=2\cos \frac{2j\pi}{n+1}$ for some $1 \le j \le n$ and $j \not= \frac{n+1}{2}$.

\underline{\em Case 1:} $\lambda=-2$ (in this case $n$ must be odd). By similar arguments as in the proof of Lemma \ref{twist}, we have  $\phi_K(x,y) \not= 0$ if $x \in \BR$.

\underline{\em Case 2:} $\lambda=2\cos \frac{2j\pi}{n-1}$ for some $1 \le j \le n-2$ and $j \not= \frac{n-1}{2}$. Then $S_{n-1}(\lambda)=1$ and $S_{n-2}(\lambda)=0$, hence $\phi_K(x,y)=-(y+1-x^2)$.

Suppose $\phi_K(x,y)=0$. Then $y=x^2-1$ and $\lambda=y^2-yx^2+2x^2-2=x^2-1$. This cannot occur if $x^2-1 \ge 2$, since $\lambda<2$. Hence $\phi_K(x,y) \not= 0$ if $x^2 \ge 3$.

\underline{\em Case 3:} $\lambda=2\cos \frac{2j\pi}{n+1}$ for some $1 \le j \le n$ and $j \not= \frac{n+1}{2}$. Then $S_{n-1}(\lambda)=-1$ and $S_{n-2}(\lambda)=-\lambda$, hence $\phi_K(x,y)=y+1-x^2+\lambda$.

Suppose $\phi_K(x,y)=0$. Then $y=x^2-\lambda-1$ and $\lambda=y^2-yx^2+2x^2-2=\lambda^2+\lambda (2-x^2)+x^2-1$, i.e. $\lambda^2-\lambda(x^2-1)+x^2-1=0$. This equation does not have any real solution $\lambda$ if $1<x^2<5$. Hence $\phi_K(x,y) \not= 0$ if $1 <x^2<5$.

In all cases, $\phi_K(x,y) \not= 0$ when $\frac{T_n(\lambda)-2}{\lambda-2}=1$ and $3 \le x^2 <5$. The lemma follows.
\end{proof}

\begin{remark}
Lemma~\ref{lem:twist-knot} gives another proof of \cite[Theorem~1.2]{Mo} without using the irreducibility of $\phi_{J(2,2n)}(2,y) \in \BZ[y]$ proved in \cite{HS2}. 
\end{remark}

\subsubsection{\textbf{The case} $J(2m+1,2n)$}\label{subsection:(2m+1,2n)} Let $K=J(2m+1,2n)$. Suppose $\rho$ is a non-abelian representation. By Propositions \ref{prop2} and \ref{prop3}, the coefficient of the highest degree term of $\Delta_{K,\rho}(t)$ is $\frac{T_m(y)-2}{y-2}$ if $n>0$, and is $\frac{T_{m+1}(y)-2}{y-2}$ if $n<0$. We want to show that for $m>1$, we have $\frac{T_m(y)-2}{y-2} \not= 1$ when $\phi_K(2,y)=0$. This will imply that for any parabolic representation $\rho$, $\Delta_{K,\rho}(t)$ is monic if and only if $K=J(1,2n)$, or $K=J(3,2n)$ and $n>0$.

The key point of the proof of the following lemma is to apply Proposition \ref{al}.

\begin{lemma}
Suppose $x=2$. For $m>1$, one has $\gcd(\phi_K(2,y),\frac{T_m(y)-2}{y-2}-1)=1$ in $\BC[y]$.
\end{lemma}

\begin{proof}
Suppose $T_m(y)=y$ and $y \not= 2$. Then $\gamma_+^m + \gamma_-^m=\gamma_+ + \gamma_-$, i.e., $\gamma_+^{m-1}=1$ or $\gamma_+^{m+1}=1$. It implies that $y=-2$, or $y=2\cos \frac{2j\pi}{m-1}$ for some $1 \le j \le m-2$ and $j \not=  \frac{m-1}{2}$, or $y=2\cos \frac{2j\pi}{m+1}$ for some $1 \le j \le m$ and $j \not= \frac{m+1}{2}$. In all cases, $y \in \BR$ and $y<2$. Proposition \ref{al} then implies that $\phi_K(2,y) \not= 0$. The lemma follows.
\end{proof}

The arguments in Subsections~\ref{subsection:(2m,2n)},~\ref{subsection:(2,2n)} 
and \ref{subsection:(2m+1,2n)} show Theorem~\ref{thm:main-theorem} (2). 
This completes the proof of Theorem~\ref{thm:main-theorem}. 

\vspace{2mm}

\noindent
\textit{Acknowledgements}. 
The first author was partially supported by 
Grant-in-Aid for Scientific Research (No.\,23540076), 
the Ministry of Education, Culture, Sports, Science 
and Technology, Japan. The second author would like to thank T.T.Q. Le for helpful discussions.


\begin{thebibliography}{99999}


\bibitem[BZ]{BZ} G. Burde, H. Zieschang, {\em Knots}, de Gruyter Stud. Math., vol. 5, de Gruyter, Berlin, 2003.

\bibitem[Cr]{Cr} R. Crowell, {\em Genus of alternating link types}, Ann. of Math. (2) \textbf{69} (1959), 258--275.

\bibitem[DFJ]{DFJ} N. Dunfield, S. Friedl, and N. Jackson, 
{\em Twisted Alexander polynomials of hyperbolic knots}, 
Experimental Math. {\bf 21} (2012), 329--352.


\bibitem[FK]{FK06-1}
S. Friedl and T. Kim,
\textit{The Thurston norm, fibered manifolds and twisted Alexander polynomials},
Topology {\bf 45} (2006), 929--953.

\bibitem[FKK]{FKK11-1}
S. Friedl, T. Kim and T. Kitayama,
\textit{Poincar\'{e} duality and degrees of twisted Alexander polynomials}, 
Indiana Univ. Math. J. {\bf 61} (2012), 147--192.

\bibitem[FV1]{FV10-1}
S. Friedl and S. Vidussi,
\textit{A survey of twisted Alexander polynomials},
The Mathematics of Knots: Theory and Application
(Contributions in Mathematical and Computational Sciences),
eds. Markus Banagl and Denis Vogel (2010), 45--94.

\bibitem[FV2]{FV11-1}
S. Friedl and S. Vidussi, \textit{Twisted Alexander polynomials
detect fibered $3$-manifolds}, Ann. of Math. 
{\bf 173} (2011), 1587--1643. 

\bibitem[FV3]{FV12-1}
S. Friedl and S. Vidussi,
\textit{The Thurston norm and twisted Alexander polynomials},
arXiv:1204.6456.

\bibitem[GKM]{GKM05-1}
H. Goda, T. Kitano and T. Morifuji,
\textit{Reidemeister torsion, twisted Alexander polynomial
and fibered knots},
Comment. Math. Helv. {\bf 80} (2005), 51--61.

\bibitem[HSW]{HSW10-1}
J. Hillman, D. Silver and S. Williams, \textit{On reciprocality of
twisted Alexander invariants}, Algebr.~Geom.~Topol. {\bf 10} (2010),
1017--1026.

\bibitem[HS1]{HS2} J. Hoste and P. Shanahan, 
{\em Trace fields of twist knots}, J. Knot Theory Ramifications \textbf{10} (2001), 625--639.

\bibitem[HS2]{HS} J. Hoste and P. Shanahan, {\em A formula for the $A$-polynomial of twist knots}, 
J. Knot Theory Ramifications {\bf 14} (2005), 91--100.

\bibitem[KmM]{KM} T. Kim and T. Morifuji, {\em Twisted Alexander polynomials and character varieties 
of $2$-bridge knot groups}, Internat. J. Math. \textbf{23} (2012), 1250022, 24 pp.

\bibitem[KtM]{KM05-1}
T. Kitano and T. Morifuji,
\textit{Divisibility of twisted Alexander polynomials and fibered knots},
Ann. Sc. Norm. Super. Pisa Cl. Sci. (5)
{\bf 4} (2005), 179--186.

\bibitem[Le]{Le93} T.T.Q. Le, {\em Varieties of representations and their subvarieties of cohomology jumps 
for knot groups}, (Russian) Mat. Sb. {\bf 184} (1993), 57-82; 
translation in Russian Acad. Sci. Sb. Math. {\bf 78} (1994), 187-209.

\bibitem[LT]{LT} T.T.Q. Le and A.T. Tran, {\em On the AJ conjecture for knots}, 
arXiv:1111.5258. 

\bibitem[LeV]{LeV} W. LeVeques, {\em Fundamentals of number theory}, Reprint of the 1977 original, 
Dover Publications, Inc., Mineola, NY, 1996.

\bibitem[Li]{Lin01-1}
X. S. Lin,
\textit{Representations of knot groups and twisted Alexander polynomials},
Acta Math. Sin.
(Engl. Ser.) {\bf 17} (2001), 361--380.

\bibitem[Mo]{Mo} T. Morifuji, {\em On a conjecture of Dunfield, Friedl and Jackson}, C. R. Acad. Sci. Paris, Ser. I \textbf{350} (2012), 921--924.

\bibitem[Mu]{Mu} K. Murasugi, { \em On the genus of the alternating knot I, II}, J. Math. Soc. Japan \textbf{10} (1958), 94--105 and 235--248.

\bibitem[Ri]{Ri} R. Riley, {\em Nonabelian representations of $2$-bridge knot groups}, Quart. J. Math. Oxford Ser. (2) \textbf{35} (1984), 191--208.

\bibitem[Ro]{Ro} S. Roman, {\em Field theory}, Second edition, Graduate Texts in Mathematics 158, Springer, New York, 2006.

\bibitem[Th]{Th} W. Thurston. {\em Three-dimensional geometry and topology}, Vol. 1, volume 35 of Princeton Mathematical Series, Princeton University Press, Princeton, NJ, 1997. Edited by Silvio Levy.

\bibitem[Tr1]{Tr} A.T. Tran, {\em The universal character ring of some families of one-relator groups}, Algebr. Geom. Topol. \textbf{13} (2013), 2317--2333.

\bibitem[Tr2]{Tr2} A.T. Tran, {\em The universal character ring of the $(-2,2m+1,2n)$-pretzel link}, 
Internat. J. Math. \textbf{24} (2013), 1350063, 13 pp.

\bibitem[Wa]{Wada94-1}
M. Wada,
\textit{Twisted Alexander polynomial for finitely
presentable groups},
Topology {\bf 33} (1994), 241--256.
\end{thebibliography}
\end{document}